\documentclass[a4paper]{amsart}
\usepackage[backend=biber, isbn=false, doi=false, style=numeric, sortcites=true, citestyle=numeric, giveninits=true]{biblatex}
\bibliography{bibliografia.bib}

\usepackage{xfrac}
\usepackage{amssymb}
\usepackage{amsmath}
\usepackage{amsthm}
\usepackage{csquotes}
\usepackage{stmaryrd}
\usepackage{wrapfig}
\usepackage{xspace}
\usepackage{tikz}
\usetikzlibrary{cd}
\usetikzlibrary{calc}
\tikzstyle{dot}=[inner sep=1pt, fill, black, circle, draw, minimum size = 5pt]
\tikzstyle{highlight}=[inner sep=1pt, draw=black, fill=white, circle, minimum size = 9pt]
\usepackage[capitalise, noabbrev]{cleveref}
\usepackage{enumitem}
\setlist[itemize]{noitemsep}
\setlist[enumerate]{noitemsep}
\setlist[enumerate,1]{label = (\roman*)}
\usepackage{xfrac}
\usepackage{amssymb}
\usepackage{amsmath}
\usepackage{amsthm}
\usepackage{csquotes}
\usepackage{stmaryrd}
\usepackage{wrapfig}
\usepackage{xspace}
\usepackage{tikz}
\usetikzlibrary{cd}
\usetikzlibrary{calc}
\usepackage{adjustbox}
\usepackage{capt-of}
\tikzstyle{dot}=[inner sep=1pt, fill, black, circle, draw, minimum size = 5pt]
\tikzstyle{highlight}=[inner sep=1pt, draw=black, fill=white, circle, minimum size = 9pt]

\newcommand{\up}{\mathsf{Up}}

\newcommand{\cC}{\mathcal{C}}

\newcommand{\cL}{\mathcal{L}}

\newcommand{\fE}{\mathfrak{E}}

\newcommand\vari{\mathcal{V}}

\newcommand{\eqrel}{ {\sim} }
\newcommand{\neqrel}{ {\nsim} }
\newcommand{\rank}{\mathsf{rank}}

\newcommand{\theory}{\mathsf{Eq}}

\newcommand{\espace}[1]{\mathfrak{#1}}

\newcommand{\CU}{\mathcal{CU}}

\newtheorem{theorem}{Theorem}
\newtheorem{proposition}[theorem]{Proposition}

\newtheorem{definition}[theorem]{Definition}
\newtheorem{lemma}[theorem]{Lemma}
\newtheorem{corollary}[theorem]{Corollary}

\theoremstyle{definition}
\newtheorem*{notation}{Notation}
\newtheorem*{remark}{Remark}

\title{Varieties of Strictly  n-Generated Heyting Algebras}
\date{\today}
\author{Tapani Hyttinen}
\author{Davide Emilio Quadrellaro}
\email{tapani.hyttinen@helsinki.fi}
\email{davide.quadrellaro@helsinki.fi}
\address{Department of Mathematics and Statistics, University of Helsinki, P.O. Box 68 (Pietari Kalmin katu 5), 00014 Helsinki, Finland}
\thanks{D. E. Quadrellaro was supported by grant 336283 of the Academy of Finland and Research Funds of the University of Helsinki.}
\subjclass[2020]{03C05, 06D20, 08B99.}

\begin{document}
	\maketitle
		\begin{abstract}
	   For any $n<\omega$ we construct an infinite Heyting algebra $H_n$ which is $(n+1)$-generated but that contains only finite $n$-generated subalgebras. From this we conclude that for every $n<\omega$ there exists a variety of Heyting algebras which contains an infinite $(n+1)$-generated Heyting algebra, but which contains only finite $n$-generated Heyting algebras. For the case $n=2$ this provides a negative answer to a question posed by G. Bezhanishvili and R. Grigolia in \cite{bezhanishvili2005locally}.
	\end{abstract}

\section{Introduction}

A Heyting algebra $(H,\land,\lor,\to,0,1)$ is a bounded distributive lattice with an implication symbol $\to$ defined as the right adjoint of the meet, i.e. for $a,b,c\in H$
\[ a\land b \leq c \Longleftrightarrow a \leq b\to c.   \]
Heyting algebras appear naturally in many areas of mathematics. For example, the underlying lattice of open sets of a topological space forms a Heyting algebra under the subset ordering. In a topos, the partial order of the subobjects of a given object forms a Heyting algebra. Also, in algebraic logic, Heyting algebras provide a sound and complete semantics to intermediate logics, i.e. to propositional systems extending intuitionistic logic and contained in classical propositional logic.

In the present article we shall be especially interested in questions concerning infinite, finitely generated Heyting algebras. Some characterisations of finitely generated Heyting algebras were  provided in \cite{esakia1977criterion} and later in  \cite{bezhanishvili2006lattices}. \textcite{bezhanishvili2005locally} further studied locally finite varieties of Heyting algebras and they raised a question which prompted us to consider the subject of the present article, i.e. they asked whether a variety of Heyting algebras $\mathcal{V}$ is locally finite if and only if the free 2-generated $\mathcal{V}$-algebra is finite \cite[Prob. 2.4]{bezhanishvili2005locally}. Such question was studied in \cite{benjamins2020locally}, where an affirmative answer was given for varieties generated by Heyting algebras of width 2.

We show in this article that for every $n<\omega$ there exists an infinite Heyting algebra $H_n$ which is $(n+1)$-generated  but that contains only finite $n$-generated subalgebras. In turn, we prove that for every $n<\omega$ there exists a variety of Heyting algebras $\vari_n$ which contains only finite $n$-generated Heyting algebras, but which also contains some infinite $(n+1)$-generated Heyting algebra. For $n=2$ this provides a negative answer to Bezhanishvili and Grigolia's question.

We review shortly the Esakia duality \cite{esakia} between Heyting algebras and Esakia spaces, as we shall often use it in this article. We start by fixing some notation: whenever $(X,\leq)$ is a partial order and $Y\subseteq X$ we let
\begin{equation*}
Y^\uparrow = \{ x\in X  \mid  \exists y \in Y \text{ and } y \leq x \},
\qquad
Y^\downarrow = \{ x\in X  \mid  \exists y \in Y \text{ and } y \geq x  \}.
\end{equation*}
For $x\in X$, we write $x^\uparrow$ and $x^\downarrow$ for the sets $\{x\}^\uparrow$ and $\{x\}^\downarrow$ respectively. Also, given any subset $Y\subseteq X$ we write $Y^c$ for its complement $X\setminus Y$. We write $\up(X)$ for the set of upsets over a poset $X$.

We recall that an \emph{Esakia space} is a structure $\fE=(X,\tau,\leq)$, where $(X,\tau)$ is a topological space, $(X,\leq)$ a partial order and, additionally:
\begin{itemize}
	\item[(i)]  $(X,\tau)$ is compact;
	\item [(ii)] For all $x,y\in X$ such that $x\nleq y$, there is a clopen upset $U$ such that $x\in U$ and $y\notin U$;
	\item[(iii)] If $U$ is a clopen set, then also $U^\downarrow$ is clopen.
\end{itemize}
Morphisms in the category of Esakia spaces take into account both the topological and the order-theoretic structure. Given Esakia spaces $\espace{E} = (X, \tau, \leq)$ and $\espace{E'} = (X',\tau',\leq')$, a \emph{p-morphism} $f: \espace{E} \to \espace{E'}$ is a continuous map satisfying the two following conditions:
\begin{enumerate}
	\item For all $x,y \in \espace{E}$, if $x \leq y$ then $f(x) \leq' f(y)$;
	\item For all $x\in \espace{E}$ and $y' \in \espace{E'}$ such that $f(x) \leq' y'$, there exists $y \in \espace{E}$ such that $x \leq y$ and $f(y) = y'$.
\end{enumerate}

Esakia duality provides a dual categorical equivalence \emph{à la Stone} between the category of Heyting algebras with homomorphisms and the category of Esakia spaces with p-morphisms. We review the two functors of this duality. On the one hand,  a Heyting algebra is associated to its \textit{spectrum} of prime filters $\fE_H$. A \emph{prime filter} $F\subseteq H$ is a proper filter for which $x\lor y\in F$ entails $x\in F$ or $y\in F$. The set $\fE_H$ is then endowed with the topology induced by the subbasis
\[ \{ \phi(a) \,|\, a\in H \} \cup \{ \phi(a)^c \,| \, a\in H \} \]
where  $\phi(a) = \{ F\in \fE_H \,|\, a\in F \}$. A homomorphism $h:H\to H'$ is then associated to its dual p-morphism $h_*:\fE_{H'}\to \fE_H$ defined by $h_*(F)=h^{-1}[F]$.   On the other hand, an Esakia space $\fE$ is associated to the Heyting algebra of clopen upsets $\mathcal{CU}(\fE)$,  where the  Boolean operations are interpreted in the obvious way as intersection and union, and implication is defined by $U\to V:=((U\setminus V)^\downarrow)^c$. Similarly, a p-morphism $p:\fE\to \fE'$ is associated to its dual homomorphism $p^*:\mathcal{CU}(\fE')\to \mathcal{CU}(\fE)$ defined by $p^*(U)=p^{-1}[U]$. We refer the reader to \cite{esakia} for details on such constructions. In the present article we will generally identify elements of a Heyting algebra with clopen upsets of its dual Esakia space and we will generally think of elements of an Esakia space as prime filters over the dual Heyting algebra.

\section{Poset Colourability}

\begin{definition} Let $P$ be an arbitrary poset and $G\subseteq \up(P)$.
	\begin{enumerate}
		\item We define by recursion:
		\begin{align*}
		x \eqrel^G_0 y \; &\;\Longleftrightarrow \; \forall g\in G \; (x\in g \iff y\in g)  \\
		x\eqrel^G_{n+1} y \; &\;\Longleftrightarrow \; \forall z\geq x \; \exists k\geq y \; (z\eqrel^G_{n} k) \; \land \; \forall k\geq y \; \exists z\geq x \; (z\eqrel^G_{n} k)
		\end{align*}
		and we let $ \eqrel^G_{\omega}  =\bigcap_{n\in\omega} \eqrel^G_{n}  $.		
		
		\item For any $x\in P$, the \emph{n-type of $x$ over $G$} is the set $t_n(x/G)$ of all points $y\in P$ such that $x\eqrel^G_n y$. The \emph{$\omega$-type of $x$ over $G$} is the set $t(x/G)$ of all points $y\in P$ such that $x\eqrel^G_\omega y$. 
		
		\item For $X\subseteq P$ we denote by $X/\eqrel^G_n$ the set of $n$-types $\{ t_n(x/G) \mid x\in X \}$ and by $X/\eqrel^G_\omega$ the set of  $\omega$-types $\{ t(x/G) \mid x\in X \}$.
		
		\item We say that a point $x\in P$ is \emph{$ G $-isolated} if $x\eqrel^G_{\omega}y$ entails $x=y$, namely if $t(x/G)=\{x\}$.
				
		\item We say that $P$ is \emph{$G$-coloured}, or \emph{coloured by $G$}, if every $x\in P$ is $G$-isolated.
	\end{enumerate}	
\end{definition}

\begin{lemma}\label{refinement}
	For any $G\subseteq \up(P)$, $k<\omega$, the equivalence relation $\eqrel^G_{k+1}$ refines $\eqrel^G_{k}$.
\end{lemma}
\begin{proof}
	We proceed by induction. For $k=0$ we assume  $x \neqrel^G_0 y$, then there is without loss of generality some $g_i\in G$ such that $x\in g_i$, $y\notin g_i$. Since $g_i$ is an upset it follows that $z\in g_i$ for every $z\geq x$ but $y\notin g_i$, showing  $x \neqrel^G_1 y$.
	
	For $k=m+1$ assume  $x \neqrel^G_{k} y$, then there is without loss of generality some $z\geq x$ such that for all $w\geq y$, $z\neqrel^G_{k-1} w$. Hence, by induction hypothesis, $z\neqrel^G_{k}w$ and therefore $x \neqrel^G_{k+1} y$.
\end{proof}

\begin{notation}
	We are especially interested in the case when the poset $P$ is an Esakia space. We denote by $H,K,\dots$ Heyting algebras and by $\fE_H,\fE_K,\dots$ their Esakia duals. Without loss of generality, we think of Heyting algebras as consisting of clopen upsets over their Esakia duals and we think of Esakia spaces as consisting of prime filters over their dual Heyting algebras. Also, if $\phi(x_0,\dots,x_{k-1})$ is a first-order term in the language of Heyting algebras and $G\subseteq H$, then $\phi(g_0,\dots,g_{k-1})$ is an element of $H$ and can also be viewed as a clopen upset of $\fE_H$. Finally, we remark that a subalgebra of a Heyting algebra $H$ is simply a substructure of $H$ in the language of Heyting algebras $\cL=\{\land,\lor,\to,0,1  \}$. For any $G\subseteq H$ we write $\langle G \rangle $ for the smallest subalgebra of $H$ containing $G$.
\end{notation}

\begin{definition}
	Let $\fE_H$ be an Esakia space and $G=\{g_i \mid i< k\}\subseteq H$ be a finite set of clopen upsets.
	\begin{enumerate}
		\item Let $\phi:=\phi(x_0,\dots,x_{k-1})$ be a first-order term in the language of Heyting algebras. The \emph{implication rank} of $\phi$ is defined as follows:
		\begin{enumerate}
			\item If $\phi(x_0,\dots,x_{k-1})=x_i$ for some $i<k$, $\phi(x_0,\dots,x_{k-1})=0$ or $\phi(x_0,\dots,x_{k-1})=1$ then   $\rank(\phi)=0$;
			\item $\rank(\psi\land \chi)=\text{max}\{\rank(\psi),\rank(\chi)  \}$;
			
			\item $\rank(\psi\lor \chi)=\text{max}\{\rank(\psi),\rank(\chi)  \}$;
			
			\item $\rank(\psi\to \chi)=\text{max}\{\rank(\psi),\rank(\chi)  \}+1$.
		\end{enumerate}
		
		\item The \emph{implication rank of a clopen upset} $ U\in \langle G \rangle$ is the least implication rank of a term $\phi$ such that $\phi(g_0,\dots,g_{k-1})=U$. For any clopen upset $U$ we write $\phi_U(\vec{x})$ for this term of minimal rank and $\phi_U(\vec{g})$ for its interpretation.
	\end{enumerate}
\end{definition}

The following lemma and the subsequent theorem  generalize \cite[Prop. 34]{grilletti2023esakia} and they are essentially a reformulation of \cite[Thm. 3.1.5]{bezhanishvili2006lattices}. The relation between the implication rank of intuitionistic propositional formulas and the existence of a back-and-forth system of corresponding length was established in \cite[Thms. 4.7-4.8]{visser1996uniform}.

\begin{lemma}\label{lemmaquotient}  Let  $H$ be a Heyting algebra, let $\fE_H$ be its dual Esakia space and let $x,y\in \fE_H$. The following condition holds for any finite $G\subseteq H$:		
	\begin{align*}
	x \eqrel^G_{n} y & \Longleftrightarrow  \forall U \in \langle G \rangle \text{ with } \rank(U)\leq n: \; (x \in U \iff y \in U).
	\end{align*}
\end{lemma}
\begin{proof} We fix an enumeration $G=\{g_i \mid i< k\}$ and we proceed as follows.
	
	 $(\Rightarrow)$	We start by showing the first direction by induction on $n$.
	\begin{itemize}
		\item  If $x \eqrel^G_{0} y$ and $\rank(U)=0$, then $\phi_U$ is a Boolean combination of atomic terms. If $\phi_U=0$ or $\phi_U=1$ then $U=\emptyset$ or $U=\fE_H$ and the claim follows immediately.  Now, if $\phi_U=g_i$ for some $i<k$, we have by definition that $x\in g_i$ if and only if $y\in g_i$. If $\phi_U= \alpha\land\beta$ or  $\phi_U=\alpha\lor \beta$, then the claim follows immediately by the induction hypothesis. 
		
		\item If $x \eqrel^G_{m+1} y$ and $\rank(U)\leq m+1$,  then we proceed by induction on the complexity of $\phi_U$. If $\phi_U$ is atomic or a Boolean combination of atomic terms, then  $\rank(\phi_U)=0$. By \cref{refinement} we have that $x \eqrel^G_{0} y$, hence the claim follows by the previous item.
		
		If $\phi_U=\alpha\to \beta$, then $\phi_U= ((\alpha\setminus \beta)^{\downarrow})^c$ for some $\alpha,\beta$ with $\rank(\alpha)\leq m$, $\rank(\beta)\leq m$. Suppose $x\notin ((\alpha\setminus \beta)^{\downarrow})^c$, then there is some $z\geq x$ such that $z\in \alpha\setminus \beta$ and, since $x \eqrel^G_{m+1} y$, there is some $k\geq y $ such that $z \eqrel^G_{m} k$. By induction hypothesis  $k\in \alpha\setminus \beta$, showing  $y\notin ((\alpha\setminus \beta)^{\downarrow})^c$. The converse direction is analogous.
	\end{itemize}

	\noindent $(\Leftarrow)$ We prove the converse direction by induction on $n$. \begin{itemize}
		\item Let $n=0$ and suppose without loss of generality that there is some $g_i\in G$ such that $x\in g_i$, $y\notin g_i$. Since $\rank(g_i)=0$ the claim follows immediately.
		
		\item Let $n=m+1$. If $x\neqrel^G_{m+1} y $  then  (without loss of generality) there is $z\geq x$ such that for all $k\geq y$ we have $z\neqrel^G_{m} k$. By induction hypothesis, for every $k\geq y$, there is either a clopen upset $\psi_k\in \langle G \rangle$ such that $z\in \psi_k$ and $k\notin \psi_k$, or a clopen upset $\chi_k\in \langle G \rangle$ such that $z\notin \chi_k$ and $k\in \chi_k$, with $\rank(\psi_k),\rank(\chi_k)\leq m$. We let
		\begin{align*}
			I_0&:=\{k\geq y \;|\;z\in \psi_k, k\notin \psi_k    \} \\
			I_1&:=\{k\geq y \;|\;z\notin \chi_k, k\in \chi_k    \}
		\end{align*}
		and then we define the upset $Z:= \bigcap_{k\in I_0} \psi_k \to \bigcup_{k\in I_1} \chi_k$, namely $Z:=(( \bigcap_{k\in I_0} \psi_k \setminus  \bigcup_{k\in I_1} \chi_k )^\downarrow)^c $.
		
		Now, since for every $k\in I_0$ we have $z\in \psi_k$ and for every $k\in I_1$ we have $z\notin \chi_k$, it follows that $z\in \bigcap_{k\in I_0} \psi_k \setminus  \bigcup_{k\in I_1} \chi_k$ whence $x\notin (( \bigcap_{k\in I_0} \psi_k \setminus  \bigcup_{k\in I_1} \chi_k )^\downarrow)^c  $.
		
		If $y\notin (( \bigcap_{k\in I_0} \psi_k \setminus  \bigcup_{k\in I_1} \chi_k )^\downarrow)^c$ then  for some $k\geq y$ we have $k\in \bigcap_{k\in I_0} \psi_k \setminus \bigcup_{k\in I_1} \chi_k$. Now, if $k\in I_0$ then $k\notin \psi_k$, whence  $ k\notin \bigcap_{k\in I_0} \psi_k \setminus \bigcup_{k\in I_1} \chi_k$. Otherwise, if $k\in I_1$ then $k\in \chi_k$, whence  $ k\notin \bigcap_{k\in I_0} \psi_k \setminus \bigcup_{k\in I_1} \chi_k$. Since $I_0\cup I_1=y^\uparrow$ this is a contradiction. This shows that the upset $Z$ separates $x$ and $y$.
		
		It remains to show that $Z$ is a clopen upset of rank $\leq m+1$. In fact, since $I_0$ and $I_1$ are not necessarily finite, it does not follow immediately that $Z$ is defined by a term over $G$. However, it follows by the induction hypothesis that the number of upsets $V$ of $\rank(V)\leq m$ is finite, hence the sets $\{\psi_k \mid k\in I_0  \}$ and $\{\chi_k \mid k\in I_1  \}$ are both finite. Thus there are finite subsets $I_0'\subseteq I_0$ and $I_1'\subseteq I_1$ such that $\rho:=(( \bigcap_{k\in I'_0} \psi_k \setminus  \bigcup_{k\in I'_1} \chi_k )^\downarrow)^c =(( \bigcap_{k\in I_0} \psi_k \setminus  \bigcup_{k\in I_1} \chi_k )^\downarrow)^c$. Thus $\rho$ is clearly a term and $\rank(\rho)\leq m+1$, completing our proof.
		\qedhere
	\end{itemize}		  
\end{proof}

\begin{remark}
	By \cref{lemmaquotient} two points $x,y\in \fE$ have the same type $t_n(x/G)=t_n(y/G)$ whenever they belong to the same clopen upsets $U\in \langle G\rangle $ of rank $\leq n$. This provides us with an alternative way to think about a type $t(x/G)$, which can be also identified with the collection of clopen upsets $U\in \langle G\rangle $ of rank $\leq n$ for which $x\in U$. The same considerations apply to $\omega$-types.
\end{remark}

\begin{theorem}\label{duality-result}
	Let $H$ be a Heyting algebra and let $G\subseteq H$ be finite, then $H=\langle G \rangle$ if and only if $\fE_H$ is $G$-coloured.
\end{theorem}
\begin{proof}
	$(\Rightarrow)$ Given $x\nleq y $ in $\fE_H$, there is a clopen upset $U\in \mathcal{CU}(\fE)$ such that $x\in U$ and $y\notin U$. By $H=\langle G \rangle$, it follows that $U=\psi(\vec{g})$ for some term $\psi$. But then it follows immediately by \cref{lemmaquotient} that $x\neqrel^G_\omega y$.	
	
	$(\Leftarrow)$	Let $\fE_G$ be the dual of $\langle G \rangle$ and suppose $\fE_G\neq \fE_H$. Since the points of $\fE_H$ can be identified with prime filters over $H$ and the points of $\fE_G$ with prime filters over $\langle G \rangle$,  the map $p:\fE_H \twoheadrightarrow \fE_G$ obtained by letting $p:x\mapsto x\cap \langle G \rangle$  is well-defined. Also, since $p$ is exactly the Esakia dual of the inclusion map $i:\langle G\rangle \to H$, it follows immediately by Esakia duality that $p$ is a surjective p-morphism.
	
	Moreover, since $G\subseteq \langle G \rangle$, we can think of elements of $G$ also as clopen upsets over the dual Esakia space $\fE_G$, and in particular the equivalence relation $\eqrel^G_\omega$  is well-defined over  $\fE_G$. We claim that for any clopen upset $U\in \langle G\rangle$, we have that $x\in U$ if and only if $p(x)\in U$. Here notice that we think of $U$ both as an element of $\langle G \rangle$ and as a clopen upset over both $\fE_G$ and $\fE_H$. Similarly, we think of $x$ both as an element of $\fE_H$ and as a  prime filter over $H$, and of  $p(x)$ both as an element of $\fE_G$ and as a  prime filter over $\langle G \rangle$. By Esakia duality we have the following:
		\begin{align*}
			x\in U & \Longleftrightarrow U\in x \Longleftrightarrow U\in x \cap \langle G\rangle \Longleftrightarrow U\in p(x)  \Longleftrightarrow p(x)\in U. 
		\end{align*}	
		It follows by \cref{lemmaquotient} that $p$ preserves and reflects $\eqrel^G_\omega$, i.e. that $x\eqrel^G_\omega y$ if and only if $p(x)\eqrel^{G}_\omega p(y)$. By our assumption that $\fE_G\neq\fE_H$ we then have for two different $x,y\in \fE_H$ that  $p(x)=p(y)$. Hence clearly $p(x)\eqrel^G_{\omega} p(y)$ and, by the fact that $p$ reflects $\eqrel^G_\omega$ we obtain $x \eqrel^G_{\omega} y$, which together with $x\neq y$ contradicts the assumption that $\fE_H$ is $G$-coloured.	
\end{proof}	
	
\section{Strictly $n$-Colourable Esakia Spaces}

\begin{remark}
	We stress that in this article we identify natural numbers with finite ordinals, i.e. we identify each natural number $n<\omega$ with the set $\{ x\in \omega \mid x<n\}$. 
\end{remark}

\begin{definition} Let $\fE$ be an infinite Esakia space.
	\begin{enumerate}
	\item We say that $\fE$ is \emph{$k$-colourable} if there is a function $c:k\to \CU(\fE)$ such that every $x\in \fE$ is $c[k]$-isolated.
	\item We say that $\fE$ is \emph{strongly not $k$-colourable} if for any function $c:k\to \CU(\fE)$ the number of $\eqrel^{c[k]}_{\omega}$-equivalence classes is finite.
	\item We say that $\fE$ is \emph{strictly $k$-colourable} if it is $k$-colourable and strongly not $(k-1)$-colourable.
\end{enumerate}
\end{definition}

The goal of this section is to introduce a family of Esakia spaces $(\fE_n)_{n<\omega}$ such that every $\fE_n$ is strictly $(n+1)$-colourable.

\begin{definition}	
	Let $n<\omega$, we define the Esakia space $\fE_n=(P_n,\leq,\tau)$ consisting of the points $P_n:= \{x^l_{i} \mid l\leq 2^n , i<\omega \}\cup \{x_\infty\}$ which satisfy the following conditions:
	\begin{enumerate}
		\item $x^l_{i+1}\leq x^{l'}_{i}$ for $l'\neq l+1$, $i<\omega$;
		\item $x^l_{i+t}\leq x^{l'}_{i}$ for all $l,l'\leq 2^n$, $i<\omega$, $t\geq 2$;
		\item $x_\infty\leq x^{l}_{i}$ for all $l\leq 2^n$, $i<\omega$.
	\end{enumerate}
	\noindent and we let $\tau$ be the topology induced by the following subbasis
	\[ \{ (x^l_i)^\uparrow \mid l\leq 2^n, i<\omega   \}\cup \{ ((x^l_i)^\uparrow)^c \mid l\leq 2^n, i<\omega   \} \cup \{ \emptyset, (x_\infty)^\uparrow  \}.  \]
	For each $j<\omega$ we then let $L^j_n:=\{ x^l_{j} \mid l\leq  2^n \}$ and we refer to this as the \emph{j-th level} of the poset.
\end{definition}

\begin{lemma}
	For every $n<\omega$, $\fE_n$ is an Esakia space.
\end{lemma}
\begin{proof}
	We verify the three conditions of the definition of Esakia spaces.
	
	\underline{Compactness}. Suppose  $C$ is an open cover of $\fE_n$ and let $U\in C$ be such that $x_\infty\in U$. If $U=(x_\infty)^\uparrow$ then $\{(x_\infty)^\uparrow\}$ is a finite subcover of $C$ and we are done. Otherwise, since $U$ is open and $x_\infty \in U$, we have by definition of the topology that $((x^l_i)^\uparrow)^c\subseteq U$ for some $i<\omega$ and $l\leq 2^n$. Then for every $j\leq  i+1$ and $l\leq 2^n$ let $U^l_j\in C$ be such that $x^l_j\in U^l_j$. It follows that $\{((x^l_i)^\uparrow)^c\}\cup \{U^l_j \mid j\leq i+1, l\leq 2^n  \}$ is a finite open subcover of $C$.
	
	\underline{Separation.} Given any two elements $x^l_i\nleq x^{l'}_{i'}$ we have immediately that $x^{l}_{i}\in (x^{l}_{i})^\uparrow$ but $x^{l'}_{i'}\notin (x^{l}_{i})^\uparrow$. The case of $x^l_i\nleq x_\infty$ is analogous.
	
	\underline{Downwards closure}. Finally, consider a clopen $U$. If $U=U^\downarrow$ then we are done. Otherwise let $i<\omega$ be the least index such that $L^i_n\cap U\neq \emptyset$. If $|L^i_n\cap U|\geq 2$ or $x^0_i\in U$ then $ U^\downarrow= U\cup (\{ x^l_{i} \mid l\leq 2^n  \}^\uparrow)^c$, which is clearly clopen. 	Otherwise, if $L^i_n\cap U= \{x^m_i\}$ for some $m>0$, then we have $U^\downarrow= U\cup ((x^{m-1}_{i+1})^\uparrow)^c$ which is also clopen. This completes our proof.
\end{proof}

\begin{remark}
	We notice that $\fE_0$ is exactly the  Esakia dual of the Rieger-Nishimura lattice, i.e. the free Heyting algebra with one generator. Since the only 0-colourable Esakia space is the one-point poset, it is immediate to see that $\fE_0$ is 1-colourable but not 0-colourable.
\end{remark}

\begin{figure}
	\begin{tikzpicture}[scale=0.50]
	\pgfmathsetmacro{\NODESIZE}{1.3pt}
	
	-- draw left column of nodes
	\foreach \y in {0,...,5}{
		\node[draw,circle, inner sep=\NODESIZE,fill,yshift=0em] (A \y) at (-4,3*\y) {};
	}
	\foreach \y in {0,...,5}{
		\node[draw,circle, inner sep=\NODESIZE,fill,yshift=0em] (B \y) at (-2,3*\y) {};
	}
	\foreach \y in {0,...,5}{
		\node[draw,circle, inner sep=\NODESIZE,fill,yshift=0em] (C \y) at (0,3*\y) {};
	}
	\foreach \y in {0,...,5}{
		\node[draw,circle, inner sep=\NODESIZE,fill,yshift=0em] (D \y) at (2,3*\y) {};
	}
	\foreach \y in {0,...,5}{
		\node[draw,circle, inner sep=\NODESIZE,fill,yshift=0em] (E \y) at (4,3*\y) {};
	}

		-- add invisible nodes for the descending line at infinity
		\node[draw=none] (end1a) at  (-4,-2) {};
		\node[draw=none] (end2a) at  (-2,-2) {};
		\node[draw=none] (end3a) at  (0,-2) {};
		\node[draw=none] (end4a) at  (2,-2) {};
		\node[draw=none] (end5a) at  (4,-2) {};

		\node[draw,circle, inner sep=\NODESIZE,fill,yshift=0em] (end) at  (0,-3) {};

	-- draw the ladder itself
	\foreach \y in {0,...,4}{
		\pgfmathparse{\y+1}
		\draw [-] (A \y) to (A \pgfmathresult);
		\draw [-] (A \y) to (C \pgfmathresult);
		\draw [-] (A \y) to (D \pgfmathresult);
		\draw [-] (A \y) to (E \pgfmathresult);
		
		\draw [-] (B \y) to (A \pgfmathresult);
		\draw [-] (B \y) to (B \pgfmathresult);
		\draw [-] (B \y) to (D \pgfmathresult);
		\draw [-] (B \y) to (E \pgfmathresult);
		
		\draw [-] (C \y) to (A \pgfmathresult);
		\draw [-] (C \y) to (B \pgfmathresult);
		\draw [-] (C \y) to (C \pgfmathresult);
		\draw [-] (C \y) to (E \pgfmathresult);
		
		\draw [-] (D \y) to (A \pgfmathresult);
		\draw [-] (D \y) to (B \pgfmathresult);
		\draw [-] (D \y) to (C \pgfmathresult);
		\draw [-] (D \y) to (D \pgfmathresult);
		
		\draw [-] (E \y) to (A \pgfmathresult);
		\draw [-] (E \y) to (B \pgfmathresult);
		\draw [-] (E \y) to (C \pgfmathresult);
		\draw [-] (E \y) to (D \pgfmathresult);
		\draw [-] (E \y) to (E \pgfmathresult);
	}

	\draw [ loosely dashed] (A 0) to (end1a);
	\draw [ loosely dashed] (B 0) to (end2a);
	\draw [ loosely dashed] (C 0) to (end3a);
	\draw [ loosely dashed] (D 0) to (end4a);
	\draw [ loosely dashed] (E 0) to (end5a);
	
	\end{tikzpicture}\caption{The Esakia space $\fE_2$}	
	\label{poset-example1}
\end{figure}

\begin{proposition}\label{k+1-colourability}
	The Esakia space $\fE_n$ is $ (n+1) $-colourable.
\end{proposition}
\begin{proof}
	Since $n<n+1$ there is an injection $e:2^{n}+1\to \wp(n+1)$.
	 We use this injection to define a colouring over $\fE_n$. Let $c:n+1\to \CU(\fE_n)$ be defined by letting 
	\[c(k)=\{x_0^l\in \fE_n \mid     k\in e(l)  \}.\]
	By definition of the Esakia topology over $\fE_n$ it is clear that this map is a well-defined colouring of $\fE_n$. Then for any $l<l'<2^n+1$ we have by the injectivity of $e$  that $e(l)\neq e(l')$. Thus for some $k<n+1$ we have (without loss of generality) that $k\in e(l)\setminus e(l')$, which shows that $x_0^l\neqrel^{c[n+1]}_0 x_0^{l'}$. Moreover, since by construction of $\fE_n$ any two points at level $i+1$ see different elements at level $i$, it follows that whenever $(x^l_{i})_{l<2^n+1}$ all have different $i$-types, then  $(x^l_{i+1})_{l<2^n+1}$ have different $(i+1)$-types. Also, $x_\infty$ is the only point that sees all $\omega$-types. Thus, $\fE_n$ is $ (n+1) $-colourable.
\end{proof}

\begin{lemma}\label{lemma_nextlevel}
	Let $c:m\to \CU(\fE_n)$ be a colouring of $\fE_n$ and suppose
	\begin{enumerate}
		\item $|L_n^{i}/\eqrel_\omega^{c[m]}| \leq k\leq 2^n$;
		\item For all $x,y\in L_n^{i+1}$ we have $x\eqrel^{c[m]}_{0}y$.
	\end{enumerate}
	Then $|L_n^{i+1}/\eqrel_\omega^{c[m]}| \leq k\leq 2^n$.
\end{lemma}
\begin{proof}
	Fix a colouring $c:m\to \CU(\fE_n)$ and suppose (i) and (ii) hold. By construction, every element in $L_n^{i+1}$ sees every element in $L_n^{i-n}$ for $n\geq 1$, so to differentiate the types of elements in $L_n^{i+1}$ we can restrict attention to the elements in $L_n^{i}$. Let $A_j$, $j< k$ partition $L_n^{i}$ according to the type. Clearly the type of an element in $L_n^{i+1}$ is determined by the equivalence classes $A_j$ it sees. If $|A_j|\geq 2$, then by construction $L_n^{i+1}\subseteq A_j^\downarrow$, so the only equivalence classes that discriminate between elements of $L_n^{i+1}$ are those of size 1. In particular, each element from $L_n^{i+1}$ must see every equivalence class $A_j$ but at most one of size 1. Since the number of such equivalence classes is $q<k<2^n+1$, our claim follows readily.
\end{proof}

\begin{lemma}\label{long-lemma}
	Let $c:n\to \CU(\fE_n)$, then there is $j<\omega$ such that at least two elements in $L_n^j$ have the same $ \omega $-type and every element in $L_n^{q}$ for $q> j$ has the same 0-type.
\end{lemma}
\begin{proof} Let $c(l)=U_l$ for all $l<n$ and suppose without loss of generality that every $U_l$ is finite and nonempty, for otherwise $U_l=\emptyset$ or $U_l=\fE_n$ and thus does not contribute to distinguish the type of elements of $\fE_n$. For each $l<n$ we then denote by $i_l$  the least index such that $U_l\cap L^{i_l+1}_n=\emptyset$. We assume without loss of generality that $i_k\leq i_{k'}$ for $k<k'$ and we study types in the interval of levels between $L^{i_{0}}_n$ and $L_n^{i_{n}}$.
	
	We proceed by induction and we count for any $m\leq n$ how many $\omega$-types are determined by the upsets $(U_l)_{l<m}$ at level $L_n^{i_{m}}$. The claim is that for every $l<n$ the number of $\omega$-types over $c[l+1]$ of elements of $L_n^{i_{l}}$ is bounded above by $2^{l+1}$.
	
	\underline{Induction Base.} Firstly, we consider the types induced by $c[1]=\{U_0\}$.  If $i_0=0$ then it is clear that there are only two $\omega$-types at level $L^{i_0}_n$ induced by $U_0$ and $U_0^c$.  If $i_0>0$, then we notice that, by construction, the $\omega$-type of any element  $x\in L^{i_0}_n$  is fully determined by $U_0\cap (L^{i_0}_n\cup L^{i_0-1}_n)$.  Also, we remark that if $x,y\in L^{i_0-1}_n$ and $x\eqrel^{c[1]}_{0}y$ then since $L^{i_0-p}\subseteq U_0$ for any $p>1$ it follows that $x\eqrel^{c[1]}_{\omega}y$. We distinguish two cases.
	
	\underline{Case A.}  Suppose that $|L^{i_0}_n \cap U_0|>1$. By construction of $\fE_n$ we immediately conclude that elements from $L^{i_0}_n$ see the same $\omega$-types from $ L^{i_0-1}_n $, hence the only two $\omega$-types realized at level $L^{i_0}_n$ are those determined by the sets $L^{i_0}_n \cap U_0$ and $L^{i_0}_n\cap U_0^c$.
	
	\underline{Case B.} Suppose that $|L^{i_0}_n \cap U_0|=1$. Then we can  assume without loss of generality that $L^{i_0}_n \cap U_0=\{x^0_{i_0}\}$ and $L^{i_0}_n\cap U_0^c=\{ x^t_{i_0} \mid t>0  \}$. Then, by construction of $\fE_n$, it follows that $L^{i_0-1}_n \cap U_0=\{ x^t_{i_0-1} \mid t\neq 1  \}$ and $L^{i_0-1}_n\cap U_0^c=\{x^1_{i_0-1}\}$. Also, for every $m>0$ we have that $x^m_{i_0}\leq  x^1_{i_0-1}$ and $x^m_{i_0}\leq  x^t_{i_0-1}$ for some $t\neq 1$. Since the elements in $\{ x^t_{i_0-1} \mid t\neq 1  \}$ have the same $\omega$-type over $c[1]$, it follows that the only two $\omega$-types realized at $L^{i_0}_n$ are that of $x^0_{i_0}$ and the type of any $x^m_{i_0}$ for $m>0$. 
	
	\underline{Induction Step.} Suppose by induction that $c[l+1]$ determines at most $2^{l+1}$-many $\omega$-types at level $L_n^{i_{l}}$.  If $i_l=i_{l+1}=0$, then the number of $\omega$-types over $c[l+2]$ of elements in $L_n^0$ is fully determined by their types over $c[l+1]$ together with the set $U_{l+1}\cap L_n^0$, namely for $x,y\in L_n^0$ we have
	\[x\eqrel^{c[l+2]}_{\omega} y \Longleftrightarrow x\eqrel^{c[l+1]}_{\omega} y \text{ and }x\eqrel^{\{U_{l+1}\}}_{0} y \]
	which immediately gives us an upper bound of $2^{l+2}$ possible $\omega$-types over $c[l+2]$. 
	
	We can thus assume that $i_{l+1}>0$. Moreover, if $i_l<i_{l+1}$, then by induction hypothesis and \cref{lemma_nextlevel} we have that the elements of $ L^{i_{l+1}-1}_n $ realize at most $2^{l+1}$-many $\omega$-types over $c[l+1]$. Similarly, if $i_l=i_{l+1}$ then by induction hypothesis we have that  the elements of $ L^{i_{l+1}}_n $ realize at most $2^{l+1}$-many $\omega$-types over $c[l+1]$, whence, by construction of $\fE_n$, also the elements of $ L^{i_{l+1}-1}_n $ realize at most $2^{l+1}$-many $\omega$-types over $c[l+1]$. We next proceed with the same case-distinction that we considered in the base case.
	
	\underline{Case A.}  Suppose that $|L^{i_{l+1}}_n \cap U_{l+1}|>1$. By construction of $\fE_n$ we immediately conclude that elements from $L^{i_{l+1}}_n$ which see the same $\omega$-types over $c[l+1]$ from $ L^{i_{l+1}-1}_n $ also see the same $\omega$-types over $c[l+2]$ from $ L^{i_{l+1}-1}_n $. It follows that, exactly as above, the $\omega$-types over $c[l+2]$ of the elements in $L^{i_{l+1}}_n$ are determined by their $\omega$-types over $c[l+1]$ together with the set $U_{l+1}\cap L^{i_{l+1}}_n$, in the sense that for $x,y\in L^{i_{l+1}}_n$
	\[x\eqrel^{c[l+2]}_{\omega} y \Longleftrightarrow x\eqrel^{c[l+1]}_{\omega} y \text{ and }x\eqrel^{\{U_{l+1}\}}_{0} y \]
	which gives us at most $2^{l+2}$ possible $\omega$-types over $c[l+2]$ for elements of $L^{i_{l+1}}_n$.
	
	\underline{Case B.} Suppose that $|L^{i_{l+1}}_n \cap U_{l+1}|=1$. Then we can  assume without loss of generality that $L^{i_{l+1}}_n \cap U_{l+1}=\{x^0_{i_{l+1}}\}$ and $L^{i_{l+1}}_n\cap U_{l+1}^c=\{ x^t_{i_{l+1}} \mid t>0  \}$. Then, by construction of $\fE_n$, it follows that $L^{i_{l+1}-1}_n \cap U_{l+1}=\{ x^t_{i_{l+1}-1} \mid t\neq 1  \}$ and $L^{i_{l+1}-1}_n\cap U_{l+1}^c=\{x^1_{i_{l+1}-1}\}$. 
	
	By induction hypothesis and the remarks above, the elements in  $\{ x^t_{i_{l+1}-1} \mid t\neq 1  \}$ have at most $2^{l+1}$-many different $\omega$-types over $c[l+1]$ and, since they all belong to $U_{l+1}$, they have at most $2^{l+1}$-many different $\omega$-types also over $c[l+2]$. Now, for every $m>0$ we have that $x^m_{i_{l+1}}\leq  x^1_{i_{l+1}-1}$ and, also
	 \[ |\{ x^t_{i_{l+1}-1} \mid t\neq 1  \}|-1\leq |(x^m_{i_{l+1}})^\uparrow \cap \{ x^t_{i_{l+1}-1} \mid t\neq 1  \} | \leq  |\{ x^t_{i_{l+1}-1} \mid t\neq 1  \}|.  \] 
	 Thus the type of an element $x^m_{i_{l+1}}$ with $m>0$ is determined by which one (if any) of the element from $\{ x^t_{i_{l+1}-1} \mid t\neq 1  \}$ it does not see. This gives us a total of at most $2^{l+1}+1$ possible $\omega$-types. Together with the  type of $x^0_{i_{l+1}}$ we obtain at most $2^{l+1}+2\leq 2^{l+2}$ possible $\omega$-types over $c[l+2]$ of elements from $L^{i_{l+1}}_n$.
	
	Finally, it follows from the previous induction that  the number of $\omega$-types of elements in $L^{i_{n-1}}_{n}$ over $c[n]$ is bounded above by $2^{n}$. Since $\fE_n$ has width $2^{n}+1$, it follows that at least two elements of $L^{i_{n-1}}_{n}$ have the same $\omega$-type. Moreover, by choice of $i_{n-1}$, every element in $L^{q}_n$ for $q> i_{n-1}$ has the same 0-type over $c[n]$.
\end{proof}

\begin{lemma}\label{lemma11}
	Suppose two elements in  $L_n^{i}$  have the same $\omega$-type over some colouring $c:m\to\CU(\fE_n)$ and suppose for every $q> i$  elements in  $L_n^{q}$ have the same 0-type. Then every element of $L^{t}_n$ for $t\geq i+2^n+1$ has the same $\omega$-type.
\end{lemma}
\begin{proof}
	We reason as follows. If two elements among $ \{ x^l_i \mid  l<2^n+1  \} $ have the same $\omega$-type over $c[m]$, then it follows immediately by the construction of $\fE_n$ that $x^{2^n}_{i+1}\eqrel^{c[m]}_{\omega} x^j_{i+1}$ for some   $j<2^n$, since by assumption such elements also have the same 0-type over $c[m]$. Then $x^{2^n-1}_{i+2}\eqrel^{c[m]}_{\omega} x^{2^n}_{i+2}$, as these both see every type from $L_n^{i+1}$ and, for the same reason, we have that $x^{2^n-2}_{i+3}\eqrel^{c[m]}_{\omega}x^{2^n-1}_{i+3}\eqrel^{c[m]}_{\omega} x^{2^n}_{i+3}$. By proceeding in this way, we obtain that after $2^n+1$ steps every element of $L_n^{i+2^n+1}$ has the same $\omega$-type over $c[m]$. It follows that  for $t\geq i+2^n+1$ every element of $L^{t}_n$ see exactly the same $\omega$-types over $c[m]$ and thus has the same $\omega$-type, showing that the colouring $c[m]$ has no effect after level $L_n^{i+2^n+1}$.
\end{proof}

\begin{proposition}\label{strong-non-colourability}
	$\fE_n$ is strongly not n-colourable.
\end{proposition}
\begin{proof}
	This follows immediately by the previous lemmas.
\end{proof}

\section{Strictly $n$-Generated Heyting Algebras}

\begin{definition} Let $\cL$ be a first-order signature with no relation symbol, let $A$ be an infinite $\cL$-algebra,  and $k<\omega$.
	\begin{enumerate}
		\item We say that $A$ is \emph{$k$-generated} if there is a subset $G\subseteq A$ such that $|G|=k$ and $\langle G \rangle = A$.
		\item We say that $A$ is \emph{strongly not $k$-generated} if for every $G\subseteq A$ such that $|G|<k$ we have $|\langle G\rangle| <\omega$.
		\item We say that $A$ is \emph{strictly $k$-generated} if it is $k$-generated and strongly not $(k-1)$-generated.
		
		\item A \emph{variety} $\vari$ of $\cL$-algebras is a class of $\cL$-algebras which is axiomatised by equations, i.e. by formulas of the form $\forall x_0,\dots, \forall x_{n-1}( t(x_0,\dots,x_{n-1})=t'(x_0,\dots,x_{n-1}))$, where $t,t'$ are $\cL$-terms. Given a class of $\cL$-algebras $\cC$, we write $\theory(\cC)$ for the equational theory of $\cC$ and $\mathsf{Var}(\cC)$ for the class of algebras satisfying $\theory(\cC)$.
		
		\item A variety $\vari$ is \emph{locally finite} if every finitely generated algebra $A\in \vari$ is finite.
		
		\item A variety $\vari$ is \emph{$k$-finite} if every $k$-generated $A\in\vari$ is finite.
		
		\item A variety $\vari$ is \emph{regularly $k$-finite} if it is $k$-finite and, moreover, there are only finitely many $k$-generated algebras in $\vari$ up to isomorphism. [This definition is a finitary version of Definition 2.1(b) from \cite{bezhanishvili2001locally}]
	\end{enumerate}
\end{definition}

\noindent We recall two key results on varieties of algebras.  We write $\mathbb{H}$, $\mathbb{S}$ and $\mathbb{P}$ for the closure operations of taking homomorphic images, subalgebras and products respectively. We refer the reader to  \cite[Thm. 9.5]{Burris.1981} and \cite[Thm. 11.9]{Burris.1981} for their proofs. 
\begin{theorem} Let $\cL$ be a first-order signature with no relation symbol.
	\begin{enumerate}\label{birkhoff-tarski}
		\item \underline{Birkhoff}: A class of $\cL$-algebras is a variety if and only if it is closed under taking homomorphic images, subalgebras and products.
		\item \underline{Tarski}: For any class $\cC$ of $\cL$-algebras, $\mathsf{Var}(\cC)=\mathbb{HSP}(\cC)$.
	\end{enumerate}
\end{theorem}

\begin{proposition}\label{regular k-finite}
	Let $\cL$ be a finite first-order signature with no relation symbol, then a variety $\vari$ of $\cL$-algebras is $k$-finite if and only if it is regularly $k$-finite.
\end{proposition}
\begin{proof}
	The right-to-left direction is obvious, so we prove only the left-to-right. Suppose $\vari$ is $k$-finite but contains infinitely many $k$-generated algebras. We let $I\subseteq \omega$ the set of numbers $m<\omega$ such that  $\vari$ contains a $k$-generated algebra $A_m$ of size $m$ and we let  $a^m_0,\dots,a^m_{k-1}$ be its generators. Since the language $\cL$ is finite, there are only finitely many algebras of any fixed size $m<\omega$, showing that $I$ is cofinal in $\omega$. Let $\mathrm{ar}(f)$ be the arity of an operation $f\in \cL$ and recall that constants are 0-ary operations.  Let $\cL_i$ be the set of operations  arity $i<\omega$ and notice that the set $\cL^+:=\bigcup_{i>0}\cL_i$ is finite.   For any $m\in I$ there is a formula $\psi_m$ which expresses the fact that $x_0,\dots,x_{k-1}$ generate an algebra of size $< m$, which is defined as
	
{\footnotesize 	\begin{align*}
\psi_m:= \exists (y_j)_{j<m-1} \;   \Big (  &\big ( \bigwedge_{j<k} x_j=y_j  \big )\; \land \;\big ( \bigwedge_{c\in \cL_0}\;\bigvee_{j<m-1}c=y_j \big )\;\land\; \\ &\big ( \bigwedge_{f\in \cL^+} \; \bigwedge_{j_1, \dots,j_{\mathrm{ar}(f)}< m-1 }  \; \bigvee_{j<m-1} f(y_{j_1},\dots,y_{j_{\mathrm{ar}(f)}})= y_{j} \big )   \Big ).
\end{align*}}

	\noindent By the choice of $A_m$, we have in particular that $A_m\models  \neg \psi_l(a^m_0,\dots,a^m_{k-1})$ for all $l \leq m$. Since $I$ is cofinal in $\omega$, it follows that the set of formulas $\theory(\vari)\cup \{ \neg \psi_l \mid l< \omega \}$ is consistent. By compactness there are an algebra $A\in \vari$  and elements $a_0,\dots,a_{k-1}\in A$  such that  $A\models \neg \psi_l(a_0,\dots,a_{k-1})$ for all $l<\omega$, showing that $a_0,\dots,a_{k-1}$ generate an infinite subalgebra of $A$. By Birkhoff's Theorem we conclude that $\langle a_0,\dots,a_{k-1}\rangle \in \vari$, showing that $\vari$ is not $k$-finite.
\end{proof}

\noindent The next result can be seen as a finitary version of \cite[Thm. 3.2]{bezhanishvili2001locally}

\begin{lemma}\label{key.lemma-heyting}
	Suppose $\cL$ is a first-order signature with no relation symbol and let $A$ be an $\cL$-algebra such that
	\begin{enumerate}
		\item $A$	does not contain any infinite $k$-generated subalgebra;
		\item $A$ contains only finitely many $k$-generated subalgebras up to isomorphism.
	\end{enumerate}  
	Then $\mathsf{Var}(A)$ is (regularly) $k$-finite.
\end{lemma}
\begin{proof}
	By \cref{birkhoff-tarski}(ii) the variety $\mathsf{Var}(A)$ is obtained by taking products, subalgebras and quotients of $A$, in this precise order. It thus suffices to verify that properties (i) and (ii) are preserved by these operations.
	
	\underline{Products}. By the argument above we can consider without loss of generality only products of copies of $A$. Let $\kappa$ be a cardinal, we need to verify that the product $A^\kappa$ satisfies the properties (i) and (ii). 
	
	Consider a subset $G=\{g_l \mid l<k  \}\subseteq A^\kappa$ and  denote by $G_i:=\{g_l(i) \mid l<k\}$ the $i$-\emph{th} projection of $G$. Since by assumption $A$ satisfies the conditions (i) and (ii), it follows that every subalgebra $\langle G_i\rangle \subseteq A$ is finite and that there are only finitely many algebras $\langle G_i\rangle $ up to isomorphism. 
	
	For each $j<i<\kappa$, we let $\pi_{ij}:\langle G_i \rangle \to \langle G_j \rangle $ be the map defined by letting $\pi_{ij}(g_l(i))=g_l(j)$. By the conditions (i) and (ii),  there is (possibly after renumbering) some $N<\omega$  large enough so that $|\{ g_l{\upharpoonright}N \mid l<k \}|=|G|$ and   for every $i<\kappa$ there exists $j<N$ such that the map $\pi_{ij}:\langle G_i\rangle\to \langle G_j\rangle $  is an isomorphism. We then define the map  $h:\langle G \rangle\to A^N$ by letting
	\[h(g_l)= (\pi_{ij}(g_l(i)))_{i<\kappa, \\ j<N}\] 
	for every $l<k$. By our choice of the maps $(\pi_{ij})_{i<\kappa, j<N}$, it follows that this is a well-defined embedding of $\langle G \rangle$ into $A^N$. 
	
	Now, since $A^N$ is a finite product, any set $G'\subseteq A^N$ generates an algebra of size at most $\prod_{i<N}|\langle G'_i\rangle|<\omega$. Therefore $A^N$ satisfies condition (i) and, since $\langle G'\rangle$ is determined uniquely by its projections $G'_i$, it also satisfies condition (ii). Since the map $h:\langle G \rangle\to A^N$ is an embedding we then obtain that  $ \langle G\rangle $ is finite and so that $A^\kappa$  satisfies both conditions (i) and (ii).
	
	\underline{Subalgebras}. This is obvious as properties (i) and (ii) are clearly preserved under taking substructures.
	
	\underline{Homomorphic Images}. Let $f:B\twoheadrightarrow B'$ be a surjective homomorphism. For any choice of generators $f(g_0),\dots,f(g_{k-1})\in B'$ we have that $|\langle \{f(g_i )  \mid i<k \} \rangle | \leq f[ \langle \{g_i \mid i<k  \}\rangle ]$, which proves (i).  Moreover, since every $k$-generated subalgebra of $B$ has only finitely many homomorphic images up to isomorphism, it also follows that $B'$ satisfies (ii) as well.
\end{proof}

\begin{notation}
	We let $H_k:=\CU(\fE_k)$ be the dual Heyting algebra of $\fE_k$, namely the Heyting algebra of clopen upsets of $\fE_k$. We let $\vari_k:=\mathsf{Var}(H_k)$ be the variety of Heyting algebras generated by $H_k$, i.e. the class of all Heyting algebras which satisfy $\theory(H_k)$.
\end{notation}

\begin{proposition}\label{k-generated.Heyting}
	For every $k<\omega$, $H_k$ has the following properties:
	\begin{enumerate}
		\item $H_k$ is $(k+1)$-generated;
		\item $H_k$ does not contain any infinite $k$-generated subalgebra;
		\item $H_k$ contains only finitely many $k$-generated subalgebras up to isomorphism.	
	\end{enumerate}
	Thus in particular $H_k$ is a strictly $(k+1)$-generated Heyting algebra.
\end{proposition}
\begin{proof} $\;$
	\begin{enumerate}		
		\item $\fE_k$ is $(k+1)$-colourable by \cref{k+1-colourability} and so it follows by  \cref{duality-result} that $H_k$ is $(k+1)$-generated.  
		\item Any finite set of $k$-many generators $G=\{ g_i \mid i< k\}\subseteq H$ induces a colouring $c_G:k\to \CU(\fE_k)$ by letting $c_G(i)=g_i$ for all $i<k$. By \cref{strong-non-colourability} it follows that any such  colouring isolates only finitely many elements from $\fE_k$. In particular, there are only finitely many $\omega$-types under the colouring induced by $G$ and by \cref{lemmaquotient} only finitely many clopen upsets in $\langle G\rangle$, showing $\langle G \rangle$ is finite.
		
		\item Firstly, we notice that for any arbitrary $k$-colouring $c:k\to \CU(\fE_k)$, there exists by \cref{long-lemma,lemma11} a uniform upper bound to the number of $\omega$-types over $c[k]$ of elements of $\fE_k$. In fact, adopting the notation from \cref{long-lemma}, each restricted colouring $c[l+1]$ for $l<k$ isolates at most $2^{l+1}$-many elements at level $L^{i_{l}}_k$. Moreover, by \cref{lemma11}, each restricted colouring $c[l+1]$ can isolate only elements up to level $L^{i_l+2^k+1}_k$. Thus, independently of the choice of the colouring, there is a uniform bound to the number of elements of $\fE_k$ which can be isolated by a $k$-colouring.	
		
		Now, we have any $k$-generated subalgebra $H\subset H_k$ with a set of generators $G=\{g_i\mid i<k\}$ uniquely determines the colouring $c_G:k\to \CU(\fE_k)$ obtained by letting $c_G(i)=g_i$ for all $i<k$. \cref{lemmaquotient} then establishes a correspondence between the equivalence classes of $\eqrel^{c[k]}_{\omega}$ and the elements of $\langle G\rangle$. Therefore, we also obtain that there is an upper bound to the size of a $k$-generated subalgebra $H$ of $H_k$. Since there are only finitely many subalgebras of $H_k$ of bounded size, it follows that $H_k$ contains only finitely many $k$-generated subalgebras up to isomorphism.		
	\end{enumerate}
	This concludes our proof.	
\end{proof}

\begin{theorem}\label{main.theorem}
	For every $k<\omega$, the  variety of Heyting algebras $\mathcal{V}_k$   is (regularly) $k$-finite but not $(k+1)$-finite.
\end{theorem}
\begin{proof}
	By \cref{k-generated.Heyting} the algebra $H_k$ is an infinite, $(k+1)$-generated Heyting algebra which (i) does not contain any infinite $k$-generated subalgebra and (ii) contains only finitely many $k$-generated subalgebras up to isomorphism.  By \cref{key.lemma-heyting} properties (i) and (ii) are preserved by the variety operations, hence our statement follows.
\end{proof}

\noindent The following corollary follows immediately from the previous theorem and gives a negative answer to Problem 2.4 from \cite{bezhanishvili2005locally}.

\begin{corollary}
	The variety $\vari_2$ is (regularly) 2-finite but not locally finite.
\end{corollary}

 Finally, we conclude by showing that, for every $k<\omega$, the variety witnessing the difference between $k$-finiteness and $(k+1)$-finiteness can be chosen to be finitely axiomatisable.

\begin{corollary}
	For every $k<\omega$ there is a finitely axiomatisable variety of Heyting algebras $\vari$ which is (regularly) $k$-finite and contains an infinite, $(k+1)$-generated Heyting algebra.
\end{corollary}
\begin{proof}
	By \cref{main.theorem} the variety $\vari_k$ contains only finitely many $k$-generated algebras up to isomorphism. Therefore there is a number $N<\omega$ such that  every $k$-generated algebra in $\vari_k$ has size bounded by $N$. We notice that this property is expressed by the following formula, which adapts the formula $\psi_k$ from \cref{regular k-finite} to the present context:	
	
{\footnotesize 		\begin{align*}
	\theta_k= \forall (x_i)_{i<k} \exists (y_j)_{j<N} \; &  \Big (  \big ( \bigwedge_{i<k} x_i=y_i  \big ) \land \big (\bigvee_{j<N} y_j=0 \big )\land  \big ( \bigwedge_{\odot\in \{\land,\lor,\to \}} \bigwedge_{i,i'<N} \bigvee_{j<N} y_i\odot y_{i'}= y_{j} \big )   \Big ).
	\end{align*}}
	
	\noindent In particular we have that $\theory(\vari_k)\models \theta_k$. By compactness there is a finite set of equations $T\subseteq \theory(\vari_k)$ such that $T\models \theta_k$. Let $\vari$ be the set of all Heyting algebras satisfying $T$. Since $T\subseteq \theory(\vari_k)$ it follows immediately that $H_k\in \vari$. Moreover, since $T\models \theta_k$, we have that every $k$-generated algebra in $\vari$ is of size bounded by $N$. It follows that $\vari$ is a finitely axiomatisable variety which is regularly $k$-finite and contains an infinite, $(k+1)$-generated Heyting algebra.
\end{proof}
\newpage
\printbibliography

\end{document}